\documentclass[reqno,centertags, 12pt]{amsart}
\usepackage{comment}
\usepackage{color}

\usepackage{graphics}
\usepackage{amssymb,amsmath,amsfonts}%numbysec}
-\marginparwidth \oddsidemargin -9mm \evensidemargin \oddsidemargin
\usepackage{latexsym}
\advance\hoffset by 5mm

\def\@abssec#1{\vspace{.05in}\footnotesize \parindent .2in
{\bf #1. }\ignorespaces}
\newtheorem{theorem}{Theorem}[section]

\newtheorem{lemma}[theorem]{Lemma}
\newtheorem{proposition}[theorem]{Proposition}

\def \Tm {\mathbb T}

\newcommand{\be}{\mathbf e}

\allowdisplaybreaks \numberwithin{equation}{section}

\renewcommand{\be}{\begin{equation}}
\newcommand{\ee}{\end{equation}}

\title[Small scale creation for the SQG equation]{Small scale creation for solutions of the SQG equation}
\author{Siming He}
\thanks{Department of
Mathematics, Duke University, 120 Science Dr., Durham NC 27708, USA;\\
email: simhe@math.duke.edu}
\author{Alexander Kiselev}
\thanks{Department of
Mathematics, Duke University, 120 Science Dr., Durham NC 27708, USA;\\
email: kiselev@math.duke.edu}

\begin{document}

%{\Large \bf PRELIMINARY DRAFT}

\begin{abstract}
We construct examples of solutions to the conservative surface quasi-geostrophic (SQG) equation that must 
either exhibit infinite in time growth of derivatives or blow up in finite time.
\end{abstract}


\maketitle

\section{Introduction}\label{intro}

The SQG equation appears in atmospheric science, where it models evolution of the temperature on the surface
of a planet and can be derived under a number of assumptions from a more complete system of 3D rotating Navier-Stokes
equations coupled with temperature via Boussinesq approximation. In mathematical literature, the SQG equation first appeared
in \cite{CMT}, where a number of parallels with the 3D Euler equation were drawn (see \cite{MB} for more details),
and a possible singular scenario was presented.
Since then, the SQG equation has attracted attention of many researchers, in part because it appears to be perhaps the
simplest looking equation of fluid dynamics for which the global regularity vs finite time blow up question remains open.
In particular, the uniformly closing front singular scenario proposed in \cite{CMT} has been ruled out in \cite{Cord,CordF}.
More generally, one can look at the SQG equation as one member of the family of modified surface quasi-geostrophic equations,
given by
\begin{equation}\label{msqg}
\partial_t \omega +(u \cdot \nabla)\omega =0, \,\,\,u = \nabla^\perp (-\Delta)^{-1+\alpha}\omega,\,\,\,\omega(x,0)=\omega_0(x).
\end{equation}
When $\alpha=0,$ we obtain the 2D Euler equation in vorticity form; the case $\alpha =1/2$ corresponds to the SQG equation.
The range $0 < \alpha <1/2$ has been considered both in geophysical \cite{Held} and mathematical \cite{CIW} literature.
Moreover, more singular models with $1/2 < \alpha <1$ have been analyzed as well \cite{CCCGW}. For the entire $0 < \alpha < 1$
range, local regularity is known but the question whether smooth solutions can blow up in finite time remains open.
The only example of singularity formation for modified SQG equations has been recently given in \cite{KRYZ} for patch solutions
in half-plane for small $\alpha.$ While this example is suggestive, its implications for the smooth case are not clear.
In fact, surprisingly, there has been not a single example of smooth solutions to the SQG equation which exhibit infinite in time growth of
derivatives. Even though there are many such examples for the 2D Euler equation (see e.g. \cite{Yud2,Nad1,Den1,KS,Zla}),
the strongest to date example of growth in derivatives of the SQG equation is given in \cite{KN5} and it involves only
finite time growth. Part of the reason for this situation is that most of the 2D Euler growth examples involve boundary; however
the modified SQG equations have not been studied as much in the settings with boundary (see, however, \cite{CI1,CI2} for
recent advances). Moreover, smooth initial data deteriorates immediately to only H\"older regular if the support of $\omega_0$ 
contains the boundary in the
conservative modified SQG setting with natural no penetration boundary condition. The only Euler growth constructions that are done without boundaries in the periodic setting
are due to Denisov \cite{Den1} and Zlatos \cite{Zla}. The example of Denisov involves superlinear growth and can be extended
to the $0 < \alpha <1/2$ range in a straightforward manner. But it is not clear how to extend it to the SQG case since
a key part of the argument relies on control of $\|u\|_{L^\infty}$ by $\|\omega\|_{L^p}$ for some $p<\infty.$
The example of Zlatos, on the other hand, leads to exponential growth of $\nabla^2 \omega$ for smooth solutions, and
relies on representation of the velocity $u$ near origin (and under assumption of odd-odd symmetry) that goes back
to \cite{KS} and is specific to the Euler equation. Namely, one can isolate the relatively explicit ``main term" in the velocity $u$
that is of log-Lipschitz nature and dominates the rest of the Biot-Savart law in certain regimes. In the modified SQG case,
no such ``main term" behavior is expected.

In this paper, our main goal is to prove the following theorem.
\begin{theorem}\label{mainthm}
Consider the modified SQG equations \eqref{msqg} in periodic setting. For all $0 < \alpha <1,$ there exist initial data
$\omega_0$ such that
\begin{equation}\label{mainest}
{\rm sup}_{t \leq T} \|\nabla^2 \omega(\cdot, t)\|_{L^\infty} \geq  \exp (\gamma T),
\end{equation}
for all $T>0$ and constant $\gamma>0$ that may depend on $\omega_0$ and $\alpha.$
This constant can be made arbitrarily large by picking $\omega_0$ appropriately.
\end{theorem}
\it Remark. \rm A mild adjustment of our proof yields examples with exponential in time growth of $\|\omega\|_{C^{1,\gamma}}$ 
for all $1>\gamma >0$ if $\alpha \in (0,1/2]$ and $1>\gamma > 2\alpha-1$ if $\alpha \in (1/2,1).$ Our focus here is on proving 
\eqref{mainest}, so we leave details of the extension to H\"older $C^{1,\gamma}$ norms as an exercise for interested reader. 

Note that we do not prove global regularity of the solutions in these examples - solutions that blow up in finite time
will also satisfy \eqref{mainest}.% We will also see that the constant $c$ can be made arbitrarily large by adjusting $\omega_0.$

The scenario involved is the same as that of \cite{Zla}, and its geometry goes back to the Bahouri-Chemin stationary singular cross example \cite{BC}
for the 2D Euler equation. We work on $\Tm^2 = [-\pi,\pi)^2$ and consider solutions that are odd in both $x_1$
and $x_2.$ Generalizing the bounds in \cite{Zla,KS} we show that the contribution from the local part of the Biot-Savart law
that involves integration over $|y| \lesssim |x|$ region is small if $|x|$ is small and there is control over $\nabla^2 \omega.$
We then show that the ``medium" field contributions from the region $|x| \lesssim |y| \lesssim 1$ are near identical for both components
of the fluid velocity $u_1$ and $u_2$, the result replacing the ``main term" argument in the 2D Euler case.
The growth is then obtained by taking initial data with additional degeneracy and tracing
trajectories staying increasingly close to the separatices.

\section{Key estimates}\label{keyest}

In this section we prove several key estimates that we will need in the construction.
Since we will be working with solutions that are odd in both $x_1$ and $x_2,$
the Biot-Savart law for the modified SQG equation in the periodic setting is given
by (we omit constants depending on $\alpha$ and time dependence here for the sake of simplicity):
\begin{eqnarray}\label{bsu1}
u_1(x) = \int_0^\infty \int_0^\infty \left( \frac{x_2-y_2}{|x-y|^{2+2\alpha}}-\frac{x_2-y_2}{|\tilde{x}-y|^{2+2\alpha}}-\frac{x_2+y_2}{|\bar x-y|^{2+2\alpha}}
+\frac{x_2+y_2}{|x+y|^{2+2\alpha}}  \right) \omega(y) \,dy_1 dy_2, \\  \label{bsu2}
u_2(x) = -\int_0^\infty \int_0^\infty \left( \frac{x_1-y_1}{|x-y|^{2+2\alpha}}- \frac{x_1-y_1}{|\bar x-y|^{2+2\alpha}}
-\frac{x_1+y_1}{|\tilde{x}-y|^{2+2\alpha}} + \frac{x_1+y_1}{|x+y|^{2+2\alpha}}\right) \omega(y) \,dy_1 dy_2.
\end{eqnarray}
Here $\tilde{x}=(-x_1,x_2),$ $\bar x=(x_1,-x_2),$ and the function $\omega$ is extended to the entire plane by periodicity. We will later see that the integral converges absolutely at infinity if $\alpha>0.$ Near the singularity $x=y,$ the convergence is understood in the principal value sense if $\alpha \geq 1/2.$
% As we will shortly see, only the central cell - $[-\pi,\pi)^2$ region in the integrals above -
%contributes terms that are essential for the analysis of growth near the origin.
In what follows, we will denote the kernels in the integrals \eqref{bsu1}, \eqref{bsu2} by $K_1(x,y)$ and $K_2(x,y)$ respectively.

 Let $L \geq 1$ be a constant that we will eventually choose
to be large enough.
The first estimate addresses the contribution of the near field $y_1,y_2 \leq L|x|$ to the Biot-Savart law provided
that we have control of $\|\nabla^2 \omega\|_{L^\infty}.$ All the inequalities we show in the rest of this section
%, similarly to Lemma~\ref{degen}, 
assume that the solution remains smooth at times where these inequalities are derived.

\begin{lemma}\label{nearfiled}
Assume that $\omega$ is odd with respect to both $x_1$ and $x_2,$ periodic and smooth.
Take $L \geq 2,$ and suppose $L|x| \leq 1.$
Denote
\[ u_j^{near}(x) = \int_{[0,L|x|]^2} K_j(x,y) \omega(y)\,dy. \]
Then we have
\begin{equation}\label{unear}
|u_j^{near}(x)| \leq C x_j |x|^{2-2\alpha} L^{2-2\alpha} \|\nabla^2 \omega\|_{L^\infty}.
\end{equation}
\end{lemma}
\begin{proof}
Let us carry out the estimates for $u_1$ as the case of $u_2$ is similar. We need to control
\begin{eqnarray}\nonumber
\left| \int_0^{L|x|} \int_0^{L|x|} \left( \frac{(x_2-y_2)(|\tilde{x}-y|^{2+2\alpha}-|x-y|^{2+2\alpha})}{|\tilde{x}-y|^{2+2\alpha}|x-y|^{2+2\alpha}} \right. \right.
- \\ \left. \left. \label{est1} \frac{(x_2+y_2)(|x+y|^{2+2\alpha}-|\bar x-y|^{2+2\alpha})}{|\bar x-y|^{2+2\alpha}|x+y|^{2+2\alpha}} \right) \, \omega(y)\,dy_1 dy_2 \right|.
\end{eqnarray}
For the first term under the integral, we need to address the singularity where integration is understood in the principal value sense. So we estimate
the expression in \eqref{est1} by
\begin{eqnarray}\nonumber
\left| P.V.\int_0^{L|x|} dy_1 \int_0^{2x_2} dy_2  \frac{(x_2-y_2)(|\tilde{x}-y|^{2+2\alpha}-|x-y|^{2+2\alpha})}{|\tilde{x}-y|^{2+2\alpha}|x-y|^{2+2\alpha}} \omega (y) \right|
+ \\ \label{est2} \int_0^{L|x|} \int_0^{L|x|}  \frac{(x_2+y_2)(|x+y|^{2+2\alpha}-|\bar x-y|^{2+2\alpha})}{|\bar x-y|^{2+2\alpha}|x+y|^{2+2\alpha}}
\left( |\omega(y_1,y_2)|+|\omega(y_1,y_2+2x_2)| \right)\,dy_1dy_2.
\end{eqnarray}
Here we changed variable $y_2 \mapsto y_2 -2x_2$ in the remainder of the integral of the first term from \eqref{est1}.
The contribution $|\omega(y_1,y_2+x_2)|$ in the second integral comes from the rest of this term; the region of integration after the change of variable is enlarged
a little using that the integrand \eqref{est2} has fixed sign. Now in the first integral in \eqref{est2} we use that the kernel is odd
and the region of integration is symmetric with respect to $y_2=x_2$ line, and replace $\omega(y_1,y_2,t)$ by $\omega(y_1,y_2,t)-\omega(y_1,x_2,t).$ Note that
\begin{equation}\label{est3}
|\omega(y_1,y_2,t)-\omega(y_1,x_2,t)| = |\partial_{x_2} \omega(y_1,z_2,t) (y_2-x_2)| \leq |\partial^2_{x_1x_2} \omega(z_1,z_2,t) y_1 (y_2-x_2)|,
\end{equation}
where $z_2 \in (y_2,x_2)$ and $z_1 \in (0,y_1).$
We applied mean value theorem twice and used that $\partial_{x_2} \omega(0,y_2,t) \equiv 0$ for all times (since $\omega(0,y_2,t) \equiv 0$ due to oddness).
Using \eqref{est3}, we can estimate the first integral in \eqref{est2} by $\|\nabla^2 \omega\|_{L^\infty}$ 
times the following expressions; here $C$ is a constant that may change from line to line and may depend only on $\alpha:$ 
%\begin{eqnarray} \nonumber
\begin{eqnarray} \nonumber \int_0^{L|x|} dy_1 \int_0^{2x_2} dy_2
 \frac{(x_2-y_2)^2 y_1 (|\tilde{x}-y|^{2+2\alpha}-|x-y|^{2+2\alpha})}{|\tilde{x}-y|^{2+2\alpha}|x-y|^{2+2\alpha}} \leq \\ \nonumber
 C\int_0^{L|x|} dy_1 \int_0^{2x_2} dy_2 \frac{4(1+\alpha)x_1y_1^2 |\tilde{x}-y|^{2\alpha}}{|x-y|^{2\alpha}|\tilde{x}-y|^{2+2\alpha}}  \leq  \nonumber
 C x_1 \int_0^{L|x|} dy_1 \int_0^{2x_2} dy_2 \frac{1}{|x-y|^{2\alpha}} \leq \\ \nonumber  
 C x_1 \left( x_2^{2-2\alpha}+ \int_{x_2}^{L|x|} dy_1 \int_0^{x_2} dy_2 \frac{1}{|y|^{2\alpha}} \right) \leq   C x_1 \left( x_2^{2-2\alpha}+ x_2 \int_{x_2}^{L|x|} dy_1 \frac{1}{y_1^{2\alpha}} \right) \leq \\
   C x_1 (x_2^{2-2\alpha}+ x_2 L^{1-2\alpha}|x|^{1-2\alpha}) \leq 
 C x_1|x|^{2-2\alpha} L^{2-2\alpha}.\label{est4}
  \end{eqnarray}
Here in the first step we used mean value theorem and $|y_2-x_2| \leq |x-y|,$ in the second step $y_1 \leq |\tilde{x}-y|,$ and in the third 
step split the region of integration and changed variable in the long range part. 

In the second integral in \eqref{est2}, we bound $\omega(y_1,y_2)$ and $\omega(y_1,y_2+2x_2)$ using odd-odd structure by
$\|\nabla^2 \omega\|_{L^\infty}y_1y_2$ and $\|\nabla^2 \omega\|_{L^\infty}y_1(y_2+2x_2)$ respectively.
We get that this integral does not exceed
\begin{eqnarray} \nonumber
C \|\nabla^2 \omega\|_{L^\infty} \int_0^{L|x|} \int_0^{L|x|} \frac{(x_2+y_2)x_1y_1^2(y_2+2x_2)|x+y|^{2\alpha}}{|x+y|^{2+2\alpha}|\bar x -y|^{2+2\alpha}}\,dy_1dy_2 \leq  \\
C \|\nabla^2 \omega\|_{L^\infty}x_1 \int_0^{L|x|} \int_0^{L|x|} \frac{1}{|\bar x -y|^{2\alpha}}\,dy_1dy_2 \leq C \|\nabla^2 \omega\|_{L^\infty}x_1 |x|^{2-2\alpha}
L^{2-2\alpha}, \label{est5}
\end{eqnarray}
where in the first step we used the estimate for $\omega$ and mean value theorem, and in the second step $y_2+2x_2 \leq 2|\bar x-y|$
and $y_1 \leq |x+y|.$ Combining \eqref{est4} and \eqref{est5}, we get the result of the lemma.
\end{proof}

The next result records an important property of the Biot-Savart law that makes contribution of the $L|x| \leq |y| \lesssim 1$
region of the central cell to $u_1$ and $u_2$ nearly identical when $L$ is large. %Let us denote $[0,\infty) \times [0,

\begin{proposition}\label{keymed}
%Suppose that $x$ and $L$ are such that $L|x| \leq 1.$
Let $L$ be a parameter and $x$ be such that $L|x| \leq 1.$
Assume that $\omega$ is odd with respect to both $x_1$ and $x_2,$ $\omega(x) \geq 0$ in $[0,\pi)^2$, and is positive on a set
of measure greater than $(L|x|)^2.$
Let us define
\[ u_j^{med}(x) = \int_{[0,\pi)^2 \setminus [0,L|x|]^2} K_j(x,y) \omega(y)\,dy. \]
Then for all sufficiently large $L \geq L_0 \geq 2$ and $x$ such that $L|x| \leq 1$ we have that
\begin{equation}\label{umed}
1-BL^{-1} \leq -\frac{u_1^{med}(x) x_2}{x_1 u_2^{med}(x)} \leq  1 + BL^{-1},
\end{equation}
with some universal constant $B.$
\end{proposition}
\it Remark. \rm The threshold $L_0$ is a universal constant - it does not depend on $\omega$. \\
 The condition $L|x| \leq 1$ is only intended to make sure that the region of integration in $u_j^{med}$ is nontrivial.
When applying this result, $L$ will be chosen first, and $x$ will be taken small enough later.
\begin{proof}
Observe that both the positivity of $\omega$ in $[0,\pi)^2$ and the measure of the set where it is positive
is conserved by evolution, due to incompressibility and invariance of the region $[0,\pi)^2$ under trajectory map.
This point is explained in more detail below after the proof of Lemma~\ref{background}.

The bound \eqref{umed} follows from more informative pointwise bound for the Biot-Savart kernel. We will provide details for
$K_1;$ the case of $K_2$ is similar and can actually be inferred by symmetry. Bring the expression for $K_1(x,y)$ in \eqref{bsu1}
to the common denominator. The numerator will be equal to
\begin{eqnarray}\nonumber  (x_2-y_2)(|\tilde{x}-y|^{2+2\alpha} - |x-y|^{2+2\alpha})|\bar x-y|^{2+2\alpha}|x+y|^{2+2\alpha}- \\
(x_2+y_2)(|x+y|^{2+2\alpha} - |\bar x -y|^{2+2\alpha}) |\tilde{x}-y|^{2+2\alpha}|x-y|^{2+2\alpha}.\label{numk1} \end{eqnarray}
Let us first collect the terms with $y_2$ factor, and use mean value theorem to represent them as
\begin{eqnarray*} -4(1+\alpha)x_1y_1y_2\left( (z_1^2 +(x_2-y_2)^2)^\alpha |\bar x -y|^{2+2\alpha} |x+y|^{2+2\alpha} + \right. \\ \left.
(z_2^2 +(x_2+y_2)^2)^\alpha |\tilde{x}-y|^{2+2\alpha} |x-y|^{2+2\alpha} \right), \end{eqnarray*}
where $z_1^2,z_2^2$ lie between $(x_1-y_1)^2$ and $(x_1+y_1)^2.$
We can rewrite this expression as 
\begin{eqnarray*} -4(1+\alpha)x_1y_1y_2|y|^{4+6\alpha} \left(\frac{z_1^2+(x_2-y_2)^2}{|y|^2} \right)^\alpha  \times \\ \times \left(\frac{(x_1-y_1)^2+(x_2+y_2)^2}{|y|^2} \right)^{2+2\alpha}
\left(\frac{(x_1+y_1)^2+(x_2+y_2)^2}{|y|^2} \right)^{2+2\alpha}. \end{eqnarray*}
Since $|y| \geq L |x|,$ the terms with $y_2$ factor give us a contribution equal to
\begin{equation}\label{y2cont}
-8(1+\alpha)x_1y_1y_2 |y|^{4+6\alpha}\left( 1 + O(L^{-1}) \right).
\end{equation}

Now let us consider the terms with $x_2$ factor. Here we get
\begin{eqnarray}\nonumber x_2 \left( (|\tilde{x}-y|^{2+2\alpha} - |x-y|^{2+2\alpha})|\bar x -y|^{2+2\alpha}|x+y|^{2+2\alpha} - \right. \\ \left.
 (|x+y|^{2+2\alpha} -
 |\bar x-y|^{2+2\alpha})|\tilde{x} -y|^{2+2\alpha}|x-y|^{2+2\alpha} \right). \label{x2terms} \end{eqnarray}
Observe that
\begin{eqnarray} \nonumber (|x+y|^{2+2\alpha}-|\bar x -y|^{2+2\alpha}) - (|\tilde{x}-y|^{2+2\alpha} -|x-y|^{2+2\alpha}) = \\
\nonumber 4(1+\alpha) x_2y_2 (((x_1+y_1)^2 + z_3^2)^\alpha - ((x_1-y_1)^2+z_3^2)^\alpha) = \\
\label{est11} 16 \alpha (1+\alpha) x_1x_2y_1y_2
(z_3^2+z_4^2)^{\alpha-1}, \end{eqnarray}
where $z_3^2 \in  ((x_2-y_2)^2,(x_2+y_2)^2)$ and $z_4^2 \in  ((x_1-y_1)^2,(x_1+y_1)^2).$
Also,
\begin{eqnarray} \nonumber
|\bar x -y|^{2+2\alpha} |x+y|^{2+2\alpha} - |\tilde{x}-y|^{2+2\alpha} |x-y|^{2+2\alpha} = 
|\bar x -y|^{2+2\alpha} |x+y|^{2+2\alpha} - \\ \nonumber |x -y|^{2+2\alpha} |x+y|^{2+2\alpha}+|x -y|^{2+2\alpha} |x+y|^{2+2\alpha}
-|\tilde{x}-y|^{2+2\alpha} |x-y|^{2+2\alpha} = \\ 
 4(1+\alpha)x_2y_2 \left( |x+y|^{2+2\alpha} ((x_1-y_1)^2 +z_5^2)^\alpha + |x-y|^{2+2\alpha} ((x_1+y_1)^2
+z_6^2)^\alpha \right), \label{est12}
\end{eqnarray}
 where $z_5^2, z_6^2$ belong to $((x_2-y_2)^2,(x_2+y_2)^2).$
Running a straightforward computation on \eqref{x2terms} using \eqref{est11} and \eqref{est12}, we get that the $x_2$ terms are equal to
\begin{eqnarray}\nonumber
-x_2 |\bar x -y|^{2+2\alpha} |x+y|^{2+2\alpha} 16 \alpha (1+\alpha) x_1x_2 y_1y_2 (z_3^2+z_4^2)^{\alpha-1}
+   16x_1y_1x^2_2y_2 (1+\alpha)^2 \times \\ \nonumber \times (z_2^2+(x_2+y_2)^2)^\alpha  \left(|x+y|^{2+2\alpha}
((x_1-y_1)^2 +z_5^2)^\alpha +  |x-y|^{2+2\alpha} ((x_1+y_1)^2 +z_6^2)^\alpha\right) = \\
\nonumber 16(1+\alpha) x_1x_2^2 y_1y_2 \left( \alpha |y|^{2+6\alpha} \left(-1+ O(L^{-1}) \right) + 2(1+\alpha) |y|^{2+6\alpha}
\left(1+O(L^{-1})\right) \right) = \\ \label{x2cont} 16(1+\alpha)(2+\alpha)x_1 x_2^2 y_1y_2|y|^{2+6\alpha}\left(1+O(L^{-1})\right).
\end{eqnarray}
Combining \eqref{y2cont}, \eqref{x2cont} and using that $|x| \leq |y|/L,$ we obtain that the numerator \eqref{numk1} is equal to
\[ -8(1+\alpha)x_1y_1y_2|y|^{4+6\alpha} \left(1 + O(L^{-1}) \right) \]
in the region $y_1,y_2 \geq L|x|.$ Taking into account the denominator, we get that in this region
\begin{equation}\label{k1est}
K_1(x,y) = -8(1+\alpha)x_1y_1y_2 |y|^{-4-2\alpha} (1+f_1(x,y)),
\end{equation}
where $|f_1(x,y)| \leq AL^{-1}$ with some universal constant $A.$
A similar argument (or just symmetry considerations) establishes that
\begin{equation}\label{k2est}
K_2(x,y) = 8(1+\alpha)x_2y_1y_2 |y|^{-4-2\alpha}( 1+f_2(x,y)),
\end{equation}
with $|f_2(x,y)| \leq AL^{-1}.$
%Thus we obtain that pointwise, in the region $y_1,y_2 \geq L|x|,$ we have
%\[ \frac{-K_1(x,y) x_2}{x_1 K_2(x,y)} = 1 + O(L^{-1}). \]
%Given our assumptions on $\omega
Thus
\[ -\frac{u_1^{med}(x) x_2}{x_1 u_2^{med}(x)} = \frac{\int_{[0,\pi)^2 \setminus [0,L|x|]^2} y_1y_2 |y|^{-4-2\alpha} ( 1+f_1(x,y)) \omega \,dy}
{\int_{[0,\pi)^2 \setminus [0,L|x|]^2} y_1y_2 |y|^{-4-2\alpha} ( 1+f_2(x,y)) \omega \,dy}. \]
Choose $L_0$ so that $AL_0^{-1} \leq \frac14.$ Note that given our assumption that $\omega \geq 0$ in $[0,\pi)^2$ we have
\[ 1-AL^{-1} \leq \frac{\int_{[0,\pi)^2 \setminus [0,L|x|]^2} y_1y_2 |y|^{-4-2\alpha} ( 1+f_{1,2}(x,y)) \omega \,dy}
{\int_{[0,\pi)^2 \setminus [0,L|x|]^2} y_1y_2 |y|^{-4-2\alpha} \omega \,dy} \leq 1+AL^{-1}, \]
and the integral in denomminator is not zero since support of $\omega$ in $[0,\pi)^2$ has measure larger than $[0,L|x|]^2.$
%Choose $L_0$ so that $AL_0 \leq \frac14.$ 
Then a simple computation shows that \eqref{umed} follows
for every $L \geq L_0$ with a constant $B=3A.$
\end{proof}

Now we need to estimate the contribution of all cells other than the central one.
\begin{lemma}\label{farfield}
Suppose that $|x| \leq 1.$
Define
\[ u_j^{far}(x) = \int_{[0,\infty)^2 \setminus [0,\pi)^2} K_j(x,y) \omega(y) \,dy. \]
Then
\begin{equation}\label{ufar}
|u_j^{far}(x)| \leq C(\alpha) x_j \|\omega\|_{L^\infty}.
\end{equation}
\end{lemma}
\begin{proof}
Note that the estimates \eqref{k1est}, \eqref{k2est} on the Biot-Savart kernels continue to apply when $|x| \leq 1,$ and
$y \in [0,\infty)^2 \setminus [0,\pi)^2.$ Then we get that
\begin{eqnarray*}
|u_j^{far}(x)| \leq  Cx_j \int_{[0,\infty)^2 \setminus [0,\pi)^2} y_1y_2 |y|^{-4-2\alpha} |\omega(y)|\,dy  \leq \\
Cx_j \|\omega\|_{L^\infty} \int_1^\infty r^{-1-2\alpha}\,dr = C(\alpha) x_j \|\omega\|_{L^\infty}.
\end{eqnarray*}
\end{proof}

The final estimate we need is a lower bound on the absolute value of the velocity components $(-1)^j u_j,$ $j=1,2,$ near the origin provided certain assumptions on the structure
of vorticity.
\begin{lemma}\label{background}
%There exists a constant $\delta_0 \leq L_0^{-1}$ such that for all $\delta \leq \delta_0$ the following holds.
There exists a constant $1>\delta_0>0$ such that if $\delta \leq \delta_0,$ the following is true.
Let $\delta$ be a small number, $\delta \leq 1.$
Suppose, in addition to symmetry assumptions made above, that we have $1 \geq \omega_0(x) \geq 0$ on $[0,\pi)^2$ and that
$\omega_0(x)=1$ if $\delta \leq x_{1,2} \leq  \pi-\delta$. Then for all $x$ and $L \geq L_0$ such that $L |x| \leq \delta,$ we have that
\begin{equation}\label{backgroundeq}
(-1)^j u_j^{med}(x,t) \geq cx_j \delta^{-\alpha}.
\end{equation}
\end{lemma}
\begin{proof}
First, let us describe a soft consequence of the incompressibility and symmetries for properties of the solution $\omega(x,t)$
(similar to the arguments in \cite{KS}).
Note that while the solution stays smooth,
we have
\begin{equation}\label{omsol}
\omega(x,t) = \omega_0(\Phi_t^{-1}(x))
\end{equation}
 where $\Phi_t(x)$ is a smooth, invertible, measure preserving flow map defined by
\begin{equation}\label{trajmap}
 \frac{d \Phi_t(x)}{dt} = u(\Phi_t(x),t), \,\,\,\Phi_0(x)=x.
\end{equation}
In addition, it is not hard to check that odd symmetry of $\omega$ with respect to both $x_1$ and $x_2$ and periodicity imply
that $u_1$ is odd with respect to $x_1=0$ and $x_1=\pm\pi,$ and $u_2$ is odd with respect to $x_2=0$ and $x_2=\pm\pi.$
For this reason, the region $[0,\pi)^2$ is invariant under the flow map. The formula \eqref{omsol} and the assumptions on $\omega_0$ then yield that
the measure of the set in $[0,\pi)^2$ where $\omega(x,t)$ is not equal to one does not exceed $4\pi \delta$ for all $t.$

Next, observe that a consequence of the bound \eqref{k1est} is that if $L \geq L_0$ and $L|x| \leq 1,$ then for all $y_1,y_2 \geq L|x|$
we have
\begin{equation}\label{k1above}
(-1)^jK_j(x,y) \geq Cx_j y_1y_2 |y|^{-4-2\alpha}
\end{equation}
for some constant $C>0.$
Then
\begin{eqnarray*}
(-1)^j u_j^{med}(x,t) \geq Cx_j \int_{[0,\pi)^2 \setminus [0,L|x|]^2} \frac{y_1y_2}{|y|^{4+2\alpha}} \omega(y,t)\,dy \geq \\
Cx_j \int_{M\sqrt{\delta}}^1 \frac{1}{r^{1+2\alpha}}\,dr \geq cx_j \delta^{-\alpha}
\end{eqnarray*}
with some universal $c>0.$
The value of the constant $C$ here changes from expression to expression.
In the second step we used that the measure of the set where $\omega(x,t) <1$ in $[0,\pi)^2$ does not exceed $C\delta.$
We get a lower bound if we cut out of the region of integration a sector of radius $M \sqrt{\delta}$ where the value of the kernel
is largest; $M$ needs to be chosen sufficiently large but is a universal constant.
 \end{proof}

\section{Construction}\label{main}

The next lemma is parallel to the one shown in \cite{Zla}.
In the construction, we will consider the initial data that have an additional degeneracy condition on the derivatives in $x_1$ on vertical
axis. This lemma establishes that this property is preserved for the solution while it stays smooth.

\begin{lemma}\label{degen}
Suppose that in addition to being odd in $x_1$ and $x_2$ and periodic, the initial data $\omega_0$ also satisfies
$\partial_{x_1}^{2j-1}\omega_0(0,x_2)=0$ for all $x_2,$ $j=1,\dots,n.$ Then the solution $\omega(x,t),$
while it remains smooth, also satisfies $\partial_{x_1}^{2j-1}\omega(0,x_2,t)=0.$
\end{lemma}
\it Remark. \rm That all even derivatives of $\omega_0$ in $x_1$ also vanish on $x_2$ axis follows from odd symmetry.
\begin{proof}
Let us show the result for $n=1,$ the only case we use in the construction. It can be extended to arbitrary $n$ by inductive
argument. Let us differentiate the equation for $\omega$ with respect to $x_1:$
\[ \partial_t \partial_{x_1} \omega + \partial_{x_1} u_1 \partial_{x_1}\omega + u_1 \partial^2_{x_1} \omega + \partial_{x_1}u_2 \partial_{x_2}\omega +
u_2 \partial^2_{x_1x_2} \omega =0. \]
Note that $u_1$ is odd in $x_1$ and $u_2$ is even in $x_1.$ Then the third and fourth terms in the above equation vanish if $x_1=0.$
Denoting $v(x_2,t) = \partial_{x_1} \omega(0,x_2,t),$ we get that $v$ satisfies a self contained equation on the line $(0,x_2):$
\[ \partial_t v + u_2 \partial_{x_2} v + \partial_{x_1}u_1 v =0, \]
and $v(x_2,0)=0$ by assumption. Then $v(x_2,t)$ must stay zero while $\omega$ stays smooth.
\end{proof}

We are now ready to prove Theorem~\ref{mainthm}.
\begin{proof}[Proof of Theorem~\ref{mainthm}]
Let us choose the initial data $\omega_0$ as follows. First, as we already discussed, $\omega_0$ is odd with respect to both
$x_1$ and $x_2,$ $1 \geq \omega_0(x) \geq 0$  in  $[0,\pi)^2$ and it equals $1$ in this region, apart from a strip of width $\leq \delta$ along the boundary.
The parameter $\delta \leq \delta_0<1$ will be fixed later.
We also require $\partial_{x_1}\omega_0(0,x_2)=0$ for all $x_2,$ a condition that is preserved for all times while the solution
stays smooth by Lemma~\ref{degen}. Finally, we assume that in a small neighborhood of the origin of order $\sim \delta$
we have $\omega_0(x_1,x_2) = \delta^{-4} x_1^3 x_2.$ Note that $\partial_{x_1x_1}^2\omega_0(0,x_2)=0$ by oddness, so this is the
``maximal" behavior of $\omega_0$ under our degeneracy condition.

Fix arbitrary $T \geq 1;$ for small $T$ the result follows automatically as $\|\nabla^2 \omega(\cdot, t)\|_{L^\infty} \geq c\delta^{-2}$
for all times. Take $x_1^0 = e^{-T \delta^{-\alpha/2}}$ and $x_2^0 =(x_1^0)^\beta$ where $\beta$ is a parameter. %will be chosen later.
In general, we will have three parameters in our construction: $\delta,$ $L$ and $\beta.$ 
The parameters $\beta$ and $L$ are chosen to satisfy
\begin{equation}\label{betaL}
\beta =5, \,\,\,L \geq L_0 \geq 2, \,\,\,2\beta(2+B)L^{-1} \leq 1, 
\end{equation}
where $L_0$ and $B$ are universal constants from Proposition~\ref{keymed}.
Throughout the construction, we will place constraints on $\delta$  that will be consistent; we will recap these requirements at the end of the argument.
Note that
%\begin{equation}\label{om00}
\[ \omega_0(x_1^0,x_2^0) = \delta^{-4} (x_1^0)^{3+\beta} = \delta^{-4} e^{-(3+\beta)T\delta^{-\alpha/2}}. \]
%\end{equation}
Consider the trajectory $(x_1(t),x_2(t))$ originating at $(x_1^0,x_2^0).$ We will be tracking this trajectory until
either time reaches $T,$ or $x_2(t)$ reaches $x_1^0,$ or $\|\nabla^2 \omega(\cdot, t)\|_{L^\infty}$ becomes large enough to
satisfy the lower bound we seek.

Let us denote
\[ T_0 = {\rm min} \left(T, \, {\rm min} \{ t: \,x_2(t) = x_1^0\}, \, {\rm min} \{ t: \,\|\nabla^2\omega(\cdot,t)\|_{L^\infty} \geq \exp(cT) \} \right). \]
Observe that for all $t \leq T_0,$ we have $x_2(t) \leq x_1^0.$
Suppose that for some $0 \leq t_0 \leq T_0,$ we have for the first time
\begin{equation}\label{est21} |u_j^{near}(x(t_0),t_0)| + |u_j^{far}(x(t_0),t_0)| \geq L^{-1} (-1)^j u_j^{med}(x(t_0),t_0) \end{equation}
for either $j=1$ or $j=2,$ where $L \geq L_0$ is to be fixed later (we include inequality as an option in \eqref{est21} since we
could have $t_0=0$).
Note that we must have $x_1(t_0) \leq x_1^0,$ since due to \eqref{backgroundeq} we have $u_1(x(t),t) \leq 0$ for $t < t_0.$
Because of the estimates \eqref{unear}, \eqref{ufar} and \eqref{backgroundeq}, the inequality \eqref{est21} implies
that
\begin{equation}\label{backflow} C|x(t_0)|^{2-2\alpha}L^{2-2\alpha}\|\nabla^2 \omega(\cdot, t_0)\|_{L^\infty}+C\|\omega\|_{L^\infty} \geq
cL^{-1}\delta^{-\alpha}. \end{equation}
Application of \eqref{backgroundeq} requires $L|x(t_0)| \leq \delta,$ which holds if
\begin{equation}\label{par1con}
2e^{-\delta^{-\alpha/2}} L \leq \delta.
\end{equation}
Now suppose also that $\delta$ is such that
\begin{equation}\label{par2con}
 c\delta^{-\alpha/2} \geq 2CL^{3-2\alpha}
\end{equation}
and 
\begin{equation}\label{par3con127}
c\delta^{-\alpha} \geq 2C\|\omega\|_{L^\infty} L.
\end{equation}
Then \eqref{backflow} implies that
\begin{eqnarray*} \|\nabla^2\omega(\cdot, t_0)\|_{L^\infty} \geq \delta^{-\alpha/2} |x(t_0)|^{-2+2\alpha} \geq \\
\delta^{-\alpha/2}(x_1^0)^{-2+2\alpha} =  \delta^{-\alpha/2} e^{(2-2\alpha)\delta^{-\alpha/2}T}. \end{eqnarray*}
Thus the bound we seek is satisfied at $t_0$ and we are done. Therefore, from now on we can assume that for all $t \leq T_0,$
we have
\begin{equation}\label{est22}
 |u_j^{near}(x(t),t)| + |u_j^{far}(x(t),t)| \leq L^{-1} (-1)^ju_j^{med}(x(t),t)
\end{equation}
for $j=1,2.$

Next, suppose that $T_0=T.$ Then due to \eqref{est22} and \eqref{backgroundeq} we have
\[ u_1(x(t),t) \leq -(1-L^{-1})c x_1 \delta^{-\alpha} \]
for all $t \leq T.$ Also $x_1(t) \leq x_0^1$ for all $t \leq T.$ Therefore,
\[ x_1(T) \leq x_1^0 e^{-\frac{c}{2}\delta^{-\alpha}T}. \]
On the other hand, \[ \omega(x_1(T),x_2(T),T) = \omega_0(x_1^0,x_2^0) =\delta^{-4} e^{-(3+\beta)T\delta^{-\alpha/2}}. \]
Since \[ \omega(0,x_2(T),T)=\partial_{x_1}\omega(0,x_2(T),T)=0, \]
we obtain that
\[ \|\partial_{x_1x_1}^2\omega(\cdot,T)\|_{L^\infty} \geq 2\omega(x_1(T),x_2(T),T)x_1(T)^{-2} \geq \delta^{-4}e^{(c\delta^{-\alpha/2}-(3+\beta))\delta^{-\alpha/2}T}. \]
Taking $\delta$ so that
\begin{equation}\label{par3con}
c\delta^{-\alpha/2} \geq 2(3+\beta)
\end{equation}
makes sure that the lower bound we seek holds in this case, too.

It remains to consider the case where $T_0 < T$ and \eqref{est22} holds for all $t \leq T_0.$
Then $x_2(T_0)=x_1^0.$ By \eqref{est22}, for all $0 \leq t \leq T_0$ we have
\begin{equation}\label{est31} -\frac{u_1(x(t),t)}{x_1(t)} \geq -(1-L^{-1}) \frac{u_1^{med}(x(t),t)}{x_1(t)}. \end{equation}
By the estimate \eqref{umed} of Proposition~\ref{keymed}, we also have
\begin{equation}\label{est32}  -\frac{u_1^{med}(x(t),t)}{x_1(t)} \geq (1-BL^{-1}) \frac{u_2^{med}(x(t),t)}{x_2(t)}. \end{equation}
Finally, by \eqref{est22} again,
\begin{equation}\label{est33}
\frac{u_2^{med}(x(t),t)}{x_2(t)} \geq (1-L^{-1}) \frac{u_2(x(t),t)}{x_2(t)}.
\end{equation}
Combining \eqref{est31}, \eqref{est32} and \eqref{est33}, we get that
\begin{equation}\label{est34}
-\frac{u_1(x(t),t)}{x_1(t)} \geq (1-(2+B)L^{-1}) \frac{u_2(x(t),t)}{x_2(t)}
\end{equation}
due to our choice \eqref{betaL} of $L$. % is sufficiently large which we can always arrange.
Therefore
\begin{eqnarray*}
\frac{x_1^0}{x_1(T_0)} = e^{-\int_0^{T_0} \frac{u_1(x(t),t)}{x_1(t)}\,dt} \geq
e^{(1-(2+B)L^{-1})\int_0^{T_0} \frac{u_2(x(t),t)}{x_2(t)}\,dt} = \\ \left( \frac{x_2(T_0)}{x_2^0} \right)^{1-(2+B)L^{-1}} =
(x_1^0)^{(1-\beta)(1-(2+B)L^{-1})}.
\end{eqnarray*}
Here we used that $x_2^0=(x_1^0)^\beta.$
It follows that
\[ x_1(T_0) \leq (x_1^0)^{\beta(1-(2+B)L^{-1})+(2+B)L^{-1}}. \]
Similarly to the previous case, this implies that
\begin{eqnarray*} \|\partial^2_{x_1x_1} \omega(\cdot, T_0)\|_{L^\infty} \geq 2\omega(x_1(T_0),x_2(T_0),T_0) x_1(T_0)^{-2} \geq
\delta^{-4} (x_1^0)^{3+\beta-2\beta(1-(2+B)L^{-1})} = \\
\delta^{-4} (x_1^0)^{3-\beta+2\beta (2+B)L^{-1}}= \delta^{-4} e^{\delta^{-\alpha/2}(\beta-3-2\beta (2+B)L^{-1})T}
\geq \delta^{-4} e^{\delta^{-\alpha/2}T}, \end{eqnarray*}
where in the last step we used \eqref{betaL}. 
Finally, it remains to fix  $\delta \leq \delta_0$ (where $\delta_0$ is a universal constant
from Lemma~\ref{background}) so that the conditions  \eqref{par1con}, \eqref{par2con}, \eqref{par3con127} and \eqref{par3con}
are satisfied.
\end{proof}

\noindent {\bf Acknowledgement.} \rm AK acknowledges partial support of the NSF-DMS grant 1848790, and thanks
Tarek Elgindi for stimulating discussion. The authors also thank anonymous referees for careful reading of the 
manuscript and valuable suggestions to improve the paper.


\begin{thebibliography}{99}

\bibitem{BC} H. Bahouri, J.-Y. Chemin, \it \'Equations de transport relatives \'a des champs de vecteurs nonLipschitziens et m\'ecanique des
uides. (French) [Transport equations for non-Lipschitz vector
Fields and fluid mechanics], \rm  Arch. Rational Mech. Anal., {\bf 127} (1994), no. 2, 159--181

\bibitem{CCCGW} D.~Chae, P.~Constantin, D.~Cordoba, F.~Gancedo and
J.~Wu, \it Generalized surface quasi-geostrophic equations with singular velocities,
\rm Comm. Pure Appl. Math. {\bf 65} (2012), no. 8, 1037--1066

\bibitem{CMT} P.~Constantin, A.~Majda and E.~Tabak. \textit{Formation
of strong fronts in the 2D quasi-geostrophic thermal active scalar}.
Nonlinearity {\bf 7}, (1994), 1495--1533

\bibitem{CIW}
P.~Constantin, G.~Iyer, and J.~Wu, \it
Global regularity for a modified critical dissipative quasi-geostrophic equation,
\rm  Indiana Univ. Math. J. {\bf 57} (2008), no. 6, 2681--2692

\bibitem{CI1}
P.~Constantin and M.~Ignatova, \it Critical SQG in bounded domains,
\rm  Ann. PDE {\bf 2} (2016), no. 2, Art. 8, 42 pp

\bibitem{CI2}
P.~Constantin and M.~Ignatova, \it Remarks on the fractional Laplacian
with Dirichlet boundary conditions and applications, \rm IMRN 2017, no. 6, 1653-1673

\bibitem{Cord} D.~Cordoba, \textit{Nonexistence of simple hyperbolic
blow up for the quasi-geostrophic equation}, Ann. of Math.,
{\bf 148}, (1998), 1135--1152

\bibitem{CordF} D.~Cordoba and C.~Fefferman, \it Growth of solutions for QG and 2D Euler equations,
\rm  J. Amer. Math. Soc. {\bf 15} (2002), 665–670

\bibitem{CFMR}
D. Cordoba, M.A. Fontelos, A.M. Mancho, and J.L. Rodrigo,
\it Evidence of singularities for a family of contour dynamics
equations, \rm Proc. Natl. Acad. Sci. USA {\bf 102} (2005), 5949-–5952

\bibitem{Den1} S. Denisov, Infinite superlinear growth of the gradient for the two-dimensional Euler equation,
Discrete Contin. Dyn. Syst. A, 23 (2009), no. 3, 755--764

\bibitem{Held} I.~Held, R.~Pierrehumbert, S.~Garner and K.~Swanson.
\textit{Surface quasi-geostrophic dynamics}, J. Fluid Mech., 282, (1995), 1--20

\bibitem{HouLuo} G.~Luo and T.~Hou, \it
\it Toward the finite-time blowup of the 3d axisymmetric Euler equations: A numerical investigation, \rm
Multiscale Model. Simul., {\bf 12}(4) (2014), 1722--1776
%Potentially Singular Solutions of the 3D Incompressible Euler Equations, \rm preprint arXiv:1310.0497

\bibitem{KN5} A.~Kiselev and F.~Nazarov, \it A simple energy pump for the periodic 2D surface quasi-geostrophic equation, \rm
175--179, Abel Symp., 7, Springer, Heidelberg, 2012

\bibitem{KS} A.~Kiselev and V.~Sverak, \it Small scale creation for solutions of the incompressible two dimensional Euler equation, \rm Annals of Math. {\bf 180} (2014), 1205-–1220

\bibitem{KRYZ}    A.~Kiselev, L.~Ryzhik, Y.~Yao and A.~Zlatos, \it Finite time singularity for the modified SQG patch equation, \rm
 Ann. of Math. {\bf 184} (2016), no. 3, 909--948

 \bibitem{MB} A.~Majda and A.~Bertozzi, \it Vorticity and
Incompressible Flow, \normalfont Cambridge University Press, 2002

\bibitem{Nad1} N. S. Nadirashvili, \it Wandering solutions of the two-dimensional Euler equation, (Russian) \rm
Funktsional. Anal. i Prilozhen., 25 (1991), 70--71; translation in Funct. Anal. Appl., 25 (1991),
220--221 (1992)

\bibitem{Yud2} V. I. Yudovich, \it On the loss of smoothness of the solutions of the Euler equations and the
inherent instability of
flows of an ideal
fluid, \rm  Chaos, {\bf 10} (2000), 705--719

\bibitem{Zla} A.~Zlatos, \it Exponential growth of the vorticity gradient for the Euler equation on the torus, \rm Adv. Math. {\bf 268} (2015), 396--403

\end{thebibliography}
\end{document}